\documentclass{amsart}

\usepackage{geometry}
\usepackage{setspace}
\title{Remarks on automorphisms of $\CC^*\times\CC^*$ and their basins}
\date{August 15, 2008}
\author{Liz Raquel Vivas} 
\address{Department of Mathematic, University of Michigan, Ann Arbor, MI 48109}
\email{lvivas@umich.edu}

\usepackage{amscd}

\newtheorem{theorem}{Theorem}

\newtheorem{proposition}[theorem]{Proposition}

\newtheorem{conjecture}[theorem]{Conjecture}
\newtheorem{question}[theorem]{Question}

\theoremstyle{definition}
\newtheorem{defin}[theorem]{Definition}

\theoremstyle{remark}
\newtheorem{remark}[theorem]{Remark}

\newcommand{\CC}{\mathbb{C}}

\def\Aut{{\mathrm{Aut}}}

\begin{document}

\bibliographystyle{plain}

\begin{abstract}
We study basins of attraction of automorphisms of $\CC^2$ tangent to the identity that fix both axes. Our main result is that, if a well known conjecture about automorphisms of $\CC^*\times\CC^*$ holds, then there are no basins of attraction associated to the non-degenerate characteristic directions (in the sense of Hakim), and therefore we cannot find a Fatou-Bieberbach domain that does not intersect both axis with this method.
\end{abstract}

\maketitle

\section{Introduction}

A Fatou-Bieberbach domain is a proper domain of $\CC^k$ that is biholomorphic to $\CC^k$. These domains have been extensively studied, but many questions about them are still open ~\cite{rr}. 

One of these open question is the following: (\cite{rr}, p. 79)
\begin{question}
Is there a biholomorphic map from $\CC^2$ into the set $\{zw \neq 0\}$ i.e. into the complement of the union of two intersecting complex lines?
\end{question}

One classical way of constructing Fatou-Bieberbach domains (i.e. biholomorphic maps from $\CC^k$ into $\CC^k$) is to find basins of attractions of automorphisms of $\CC^k$ with a fixed point, as follows:

For $F \in \Aut(\CC^k), F(p) = p, F'(p) = A$ we define
$$
\Omega_{F,p} = \{z \in \CC^k \mid \lim_{n\to\infty}F^{n}(z) = p \}.
$$
We say $F$ has an \textit{attracting} fixed point when $A$ is a matrix with eigenvalues of modulus less than~$1$. In this case $\Omega_{F,p}$ is biholomorphic to $\CC^k$ \cite{rr}.

In the \textit{semiattracting} case (eigenvalues of modulus smaller or equal than $1$), and for automorphisms \textit{tangent to the identity} (i.e., $A = Id$) then $\Omega_{F,p}$ can also be biholomorphically equivalent to $\CC^l$ (where $l \leq k$). See \cite{ue} for semi-attracting case and \cite{hak1},\cite{hak2},\cite{we} for automorphisms tangent to the identity.

Although not all Fatou-Bieberbach domains are basins of attraction of an automorphism of $\CC^k$ \cite{w5} these are the natural source of examples (and counter-examples) for various conjectures.

One natural approach, in order to answer Question 1 positively by using these  results, is the following (for the definitions see \cite{hak1} or Section 3 below):

\medskip

\textsl{
Find an automorphism $F$ of $\CC^2$ such that:
\begin{itemize}
\item $F$ is tangent to the identity i.e. $F(0) = 0$ and $DF(0) = Id$.
\item $F$ fixes the coordinate axes i.e. $F(z,0) = (z',0)$ and $F(0,w) = (0,w')$.
\end{itemize}
Then:
\begin{enumerate}
\item There exists an attracting fixed point for $F$.
\\or
\item There exists $v$ non-degenerate characteristic direction of $F$ at the origin such that $\textrm{Re} A(v) > 0$, where $A(v)$ is the number associated to the direction $v$.
\end{enumerate}
Then, in the case of (a) we would have a basin of attraction associated to this fixed point \cite{rr}, or in the case of (b) we would a basin associated to $v$ (as in Hakim's notation) by \cite[Theorem 5.1]{hak1}. In either case, this basin $\Omega$ would be a Fatou Bieberbach domain  and $\Omega \subset \{zw \neq 0\}$.
}

\medskip

Our main result is that, assuming a well known conjecture about automorphisms of $\CC^* \times \CC^*$, this approach is not possible. 

More precisely, we assume the following:

\begin{conjecture}
If $F \in \Aut(\CC^* \times \CC^*)$, then $F$ preserves the form:
\begin{eqnarray*}
\frac{dz \wedge dw}{zw}
\end{eqnarray*}
\end{conjecture}

Assuming this, we prove:

\begin{proposition}
If Conjecture 2 is valid, and $F$ is an automorphism of $\CC^2$ tangent to the identity that fixes the coordinate axes, then (a) and (b) are both false.
\end{proposition}

Note that this does not answer Question 1 in general:  in principle, we could have basins associated to degenerate characteristic directions.

\bigskip

\noindent
{\bf Acknowledgments.} I would like to thanks Mattias Jonsson for helpful comments on a draft of this note and Berit Stens\o nes for useful conversations. Thanks also to the referee for useful remarks. Part of this work was done during a visit to the Institut Mittag-Leffler. Thanks to the Institut for its hospitality.

\section{Automorphisms of $\CC^*\times\CC^*$}

If $F$ is an automorphism of $\CC^2$ that fixes the coordinate axes, then we have that $F|_{\CC^*\times\CC^*}$ is an automorphism of $\CC^*\times\CC^*$.

The automorphism group of $\CC^* \times \CC^*$ has been studied before \cite{ni},\cite{var1},\cite{an1}, but remains mysterious. 

\begin{proposition}
Assume Conjecture 2 holds and let $F: \CC^2 \to \CC^2$ be an automorphism of $\CC^2$ such that:
\enumerate 
\item $F(0,0) = (0,0)$, $F'(0,0) = Id$
\item $F$ fixes the coordinate axes $\{zw=0\}$.

Then we can write $F$ as follows:
\begin{equation}\large{
F(z,w) = (z e^{wg(z,w)}, we^{zh(z,w)}) }
\end{equation}
where as power series we have:
\begin{eqnarray*}
g(z,w) = \sum_{\alpha+\beta \geq k} c_{\alpha,\beta}z^{\alpha}w^{\beta}\\
h(z,w) = \sum_{\alpha+\beta \geq k} d_{\alpha,\beta}z^{\alpha}w^{\beta}
\end{eqnarray*}
and
\begin{equation}
a_{\alpha-1,\beta} = -\frac{\alpha}{\beta}  b_{\alpha,\beta-1}
\end{equation}
for $k \leq \alpha + \beta \leq 2k$.
\end{proposition}

\begin{proof}
An easy computation shows that $F(z,0) = (z,0)$ since $F|_{(z,0)}$ is an automorphism of $\CC$ and $F$ is tangent to the identity. The same argument implies $F(0,w) = (0,w)$.
If we write $F(z,w) = (f_1(z,w), f_2(z,w))$ then we have:
$f_1(z,0) = z$ and $f_2(z,0) = 0$. In the same way $f_1(0,w) = 0$ and $f_2(0,w) = w$. So we get $f_1(z,w) = z + zwr(z,w)$ and $f_2(z,w) = w + zws(z,w)$. 
So, we can write:
$$
g(z,w) := \frac{\log(1 + wr(z,w))}{w}
$$
and
$$
h(z,w) := \frac{\log(1 + zs(z,w))}{z}
$$
where $g$ and $h$ will be well defined function in $\CC^2$ and therefore we have (1).
Using Conjecture 3, we will have the following relationship between $g$, $h$ and their derivatives:
\begin{eqnarray}
g_z + h_w - gh - z gh_z - w hg_w - zw g_wh_z + zw g_z h_w= 0
\end{eqnarray}
Writing $g$ and $h$ as power series and looking at the terms of degree less than $2d$ in the left term we will have (2).
\end{proof}

\section{Basins of attraction}

We recall now Hakim's notation and results for automorphisms of $\CC^2$ tangent to the identity. 

Let $F: \CC^2 \to \CC^2$ be an automorphism of $\CC^2$ tangent to the identity. We can write $F$ as power series around the origin as follows:
$$F(z,w) = (z + P_{k}(z,w) + P_{k+1}(z,w)+..., w + Q_{k}(z,w) + Q_{k+1}(z,w)+ ...)$$
where $P_{l}$ and $Q_l$ are homogeneous polynomials of degree $l$ (or identically 0).
The \textit{order} of $F$ is the smallest $k\geq 2$ such that $(P_{k}(z,w),Q_{k}(z,w))$ does not vanish identically.

\begin{defin}
A \textit{characteristic direction} is a direction $v \neq 0$ in $\CC^2$ such that 
$$(P_k(v), Q_k(v)) = \lambda v$$ 
for some $\lambda \in \CC$, where $k$ is the order of $F$.
When $\lambda \neq 0$  then we call $v$ a \textit{non-degenerate} characteristic direction. Similarly if $\lambda =0$ then $v$ is a \textit{degenerate} characteristic direction
\end{defin}

For $v \in \CC^2$ a non-degenerate characteristic direction we can assume without loss of generality $v = (1,u_0)$ with $P_{k}(1,u_0) \neq 0$.

If we define:
$$
r(u) := Q_k(1,u) - uP_k(1,u)
$$
then the non-degenerate characteristic directions of $F$ are the zeroes of the polynomial function~$r$.

To each non-degenerate characteristic direction we associate the number:
\begin{eqnarray*}
A(v) := \frac{r'(u_0)}{P_{k}(1,u_0)}
\end{eqnarray*}

Then Hakim proves \cite[Thm. 5.1 + Remark 5.3]{hak1}:

\begin{theorem}
Let $F$ be an automorphism of $\CC^2$ tangent to the identity. Let $v$ be a non-degenerate characteristic direction. Assume that $\textrm{Re} A(v) > 0$. Then there exists an invariant attracting domain $D$ in which every point is attracted to the origin along a trajectory tangent to $v$. Then the open set:
$$
\mathcal{D} = \cup_{n=0}^{\infty}F^{-n}(D)
$$
is attracted to $0$ and $\mathcal{D}$ is biholomorphic to $\CC^2$.
\end{theorem}

With this result in mind we prove Proposition 2.

\begin{proof} [Proof of Prop. 2]
We will first prove that there are no attracting fixed points.
If we assume Conjecture 2, an easy computation shows that the Jacobian of $F$ will be:
\begin{eqnarray}
JF(z,w) = e^{wg(z,w) + zf(z,w)}
\end{eqnarray}
The fixed points for $F$ are:
\begin{itemize}
\item $(0,0)$ where $DF(0,0) = Id$, therefore the origin is not attracting.
\item $(z_0,0)$ and $(0,w_0)$. For these points an easy computation shows that they are semi-attracting (or semi-repelling) fixed points (i.e. the eigenvalues are $1$ and $\lambda$); therefore not attracting either.
\item $(z_0,w_0)$ not necessarily on the axes, where $e^{w_0g(z_0,w_0)} = e^{z_0h(z_0,w_0)} = 1$. Using equation (4) we have: $JF(z_0,w_0) = 1$, therefore $(z_0,w_0)$ will not be an attracting fixed point either (in case $DF(z_0,w_0) = Id$ see Remark 7).
\\
\end{itemize}

Now we prove that (b) is not possible.
We have:

\begin{equation}\large{
F (z,w) = (z e^{wg(z,w)}, we^{zh(z,w)}) }
\end{equation}\\

We assume that the lowest degree in $g$ (and therefore in $h$) is $k$. 
Then we have:
$$
g(z,w) = \sum_{\alpha + \beta =k} c_{\alpha,\beta}z^{\alpha}w^{\beta} + h.o.t.
$$
and
$$
h(z,w) = \sum_{\alpha + \beta =k} d_{\alpha,\beta}z^{\alpha}w^{\beta} + h.o.t.
$$
where
$$
d_{\alpha-1,\beta} = -\frac{\alpha}{\beta}  c_{\alpha,\beta-1}
$$

Therefore the order of $F$ is $k+2$:
\begin{eqnarray*}
F(z,w) &=& (z e^{wg(z,w)}, we^{zh(z,w)}) \\
&=& \left(z (1 + wg(z,w) + O(w^2g^2)),  w (1 + zh(z,w) + O(z^2h^2))\right)\\
&=& \left( z + zw\sum_{\alpha + \beta =k} c_{\alpha,\beta}z^{\alpha}w^{\beta} + O(|(z,w)|^{k+3}), w + zw\sum_{\alpha + \beta =k} c_{\alpha,\beta}z^{\alpha}w^{\beta} + O(|(z,w)|^{k+3}) \right)\\
\end{eqnarray*}

We will call the lowest degree homogeneous terms $P_{k+2}$ and $Q_{k+2}$ as in Hakim's notation:
\begin{eqnarray}
P_{k+2} (z,w) &=&  \sum_{\alpha + \beta =k} c_{\alpha,\beta}z^{\alpha+1}w^{\beta+1} \\
Q_{k+2} (z,w) &=&  \sum_{\alpha + \beta =k} d_{\alpha,\beta}z^{\alpha+1}w^{\beta+1}
\end{eqnarray}

Now we want to compute characteristic directions and the numbers associated to the non-degenerate ones. (Note that we have: $(P_{k+2}(1,0), Q_{k+2}(1,0)) = (0,0)$ and $(P_{k+2}(0,1), Q_{k+2}(0,1)) = (0,0)$; therefore $(1,0)$ and $(0,1)$ are degenerate characteristic directions).

From now on we assume that $v = (1, \theta)$ is a non-degenerate characteristic direction, with $\theta \neq 0$. So  we want to solve:
\begin{eqnarray*}
r(u) &=& Q_{k+2}(1,u) - u P_{k+2}(1,u) = 0
\end{eqnarray*}
Putting this back in (5) and (6) we get:
\begin{eqnarray*}
r(u) &=& u \left( \sum_{\beta =0}^k d_{k -\beta,\beta}u^{\beta} - \sum_{\beta= 0}^k c_{k-\beta,\beta}u^{\beta +1}\right)\\
&=& u\left( d_{k,0} + \sum_{\beta=1}^k(d_{k-\beta,\beta}-c_{k-\beta+1,\beta-1}) u^\beta - c_{0,k}\lambda^{k+1}\right)
\end{eqnarray*}

Therefore the characteristic directions will be $(1,\theta)$ where $s(\theta) = 0$ for $r(u) = u s(u)$ i.e.
\begin{eqnarray*}
s(u) = d_{k,0} + \sum_{\beta=1}^k(d_{k-\beta,\beta}-c_{k-\beta+1,\beta-1}) u^\beta - c_{0,k}\lambda^{k+1}
\end{eqnarray*}
and so we have:
\begin{equation}
s(\theta) = d_{k,0} + \sum_{\beta=1}^k(d_{k-\beta,\beta}-c_{k-\beta+1,\beta-1}) \theta^\beta - c_{0,k}\theta^{k+1} = 0
\end{equation}

Now we can compute the numbers associated to each direction with the following formula:
\begin{eqnarray}
A(\theta)&=& \frac{r'(\theta)}{P_{k+2}(1,\theta)}
\end{eqnarray}

Since we know:
\begin{equation*}
r(u) = d_{k,0}u + \sum_{\beta=1}^k(d_{k-\beta,\beta}-c_{k-\beta+1,\beta-1}) u^{\beta+1} - c_{0,k}u^{k+2}
\end{equation*}
then we can easily get:
\begin{equation*}
r ' (u) = d_{k,0} + \sum_{\beta=1}^k(\beta+1) (d_{k-\beta,\beta}-c_{k-\beta+1,\beta-1})u^{\beta} - (k+2)c_{0,k}u^{k+1}
\end{equation*}

Putting this back in (9) together with (6), we have:
\begin{eqnarray}
A(\theta) &=& \frac{d_{k,0} + \sum_{\beta=1}^k(\beta + 1) (d_{k-\beta,\beta}-c_{k-\beta+1,\beta-1}) \theta^{\beta} - (k+2)c_{0,k}\theta^{k+1}}{ \sum^k_{\beta = 0} c_{k-\beta,\beta}\theta^{\beta+1}}
\end{eqnarray}

(When $k=0$ the sum from $\beta = 1$ to $k$ is the empty sum.)

By (8) we have:
\begin{equation}
 d_{k,0} = - \sum_{\beta=1}^k(d_{k-\beta,\beta}-c_{k-\beta+1,\beta-1}) \theta^\beta + c_{0,k}\theta^{k+1}
\end{equation}
and putting this in (10) we can have some cancelations:
\begin{eqnarray*}
A(\theta) &=& \frac{- \sum_{\beta=1}^k(d_{k-\beta,\beta}-c_{k-\beta+1,\beta-1}) \theta^\beta + c_{0,k}\theta^{k+1} + \sum_{\beta=1}^k(\beta + 1) (d_{k-\beta,\beta}-c_{k-\beta+1,\beta-1}) 
\theta^{\beta} - (k+2)c_{0,k}\theta^{k+1}}{ \sum_{\beta = 0}^k c_{k-\beta,\beta}\theta^{\beta+1}}
\end{eqnarray*}

So after simplifying and canceling $\theta$ (since is not equal to $0$), we have:
\begin{eqnarray}
A(\theta)&=& \frac{\sum_{\beta=1}^k\beta (d_{k-\beta,\beta}-c_{k-\beta+1,\beta-1}) \theta^{\beta} - (k+1)c_{0,k}\theta^{k+1}}{ \sum_{\beta = 0}^k c_{k-\beta,\beta}\theta^{\beta+1}}\nonumber\\
&=& \frac{\theta \left( \sum_{\beta=1}^k \beta (d_{k-\beta,\beta}-c_{k-\beta+1,\beta-1}) \theta^{\beta-1} - (k+1)c_{0,k}\theta^k\right)}{ \theta\left(\sum_{\beta = 0}^k c_{k-\beta,\beta}\theta^\beta\right)}\nonumber\\
&=& \frac{\sum_{\beta=1}^k\beta (d_{k-\beta,\beta}-c_{k-\beta+1,\beta-1}) \theta^{\beta-1} - (k+1)c_{0,k}\theta^k}{\sum_{\beta = 0}^k c_{k-\beta,\beta}\theta^\beta}
\end{eqnarray}

If the conjecture holds true, then we have:
\begin{eqnarray*}
d_{\alpha-1, \beta} = -\frac{\alpha}{\beta}c_{\alpha, \beta-1}
\end{eqnarray*}
for $k \leq \alpha + \beta - 1 \leq 2k$. So, using $\alpha = k + 1 - \beta$ we have:
\begin{eqnarray*}
d_{k - \beta, \beta} = -\frac{k-\beta +1}{\beta}c_{k- \beta+1, \beta-1}
\end{eqnarray*}
(For $k = 0$ the condition is empty, but the result still holds.)
So, back in our equation (12):
\begin{eqnarray*}
A(\theta)&=& \frac{\sum_{\beta=1}^k \beta (-\frac{k-\beta +1}{\beta}c_{k- \beta+1, \beta-1}-c_{k-\beta+1,\beta-1}) \theta^{\beta-1} - (k+1)c_{0,k}\theta^k}{\sum^k_{\beta = 0} c_{k-\beta,\beta}\theta^\beta}\\
&=& \frac{-(k+1)\sum_{\beta = 0}^k c_{k-\beta,\beta}\theta^\beta}{\sum_{\beta = 0}^k c_{k-\beta,\beta}\theta^\beta}\\
&=& -(k+1)
\end{eqnarray*}

This finishes the proof of Proposition 2.
\end{proof}\

\begin{remark}
If $(z_0,w_0)$ is a fixed point different from the origin, where $DF(z_0,w_0) = Id$, then we can use the same technique to show that the result above still holds i.e. non-degenerate characteristic directions at $(z_0,w_0)$ will have a negative number associated to them. 
\end{remark}

\end{document}